\documentclass[10pt,reqno]{amsart}
\usepackage{amsmath}
\usepackage{amscd,amsthm,amsfonts,amsopn,amssymb,mathrsfs}
\usepackage{epsfig,hyperref}

\numberwithin{equation}{section}

\newtheorem{theorem}{Theorem}[section]

\newtheorem{proposition}[theorem]{Proposition}

\newtheorem{lemma}[theorem]{Lemma}
\theoremstyle{remark}

\DeclareMathOperator{\iRe}{Re}

\def\XXint#1#2#3{{\setbox0=\hbox{$#1{#2#3}{\int}$ }
\vcenter{\hbox{$#2#3$ }}\kern-.6\wd0}}

\allowdisplaybreaks
\begin{document}

\title[Invariant Gibbs measure evolution for radial NLW on the 3d ball]{Invariant Gibbs measure evolution for the radial nonlinear wave equation on the 3d ball}
\thanks{The research of J.B. was partially supported by NSF grants DMS-0808042 and DMS-0835373 and the research of A.B. was supported by NSF under agreement No. DMS-0808042 and the Fernholz Foundation}
\date{\today}
\author{Jean Bourgain}
\address{(J. Bourgain) School of Mathematics, Institute for Advanced Study, Princeton, NJ 08540}
\email{bourgain@math.ias.edu}
\author{Aynur Bulut}
\address{(A. Bulut) School of Mathematics, Institute for Advanced Study, Princeton, NJ 08540}
\email{abulut@math.ias.edu}
\begin{abstract}
We establish new global well-posedness results along Gibbs measure evolution for the nonlinear wave equation posed on the unit ball in $\mathbb{R}^3$ via two distinct approaches.  The first approach invokes the method established in the works \cite{B1,B2,B3} based on a contraction-mapping principle and applies to a certain range of nonlinearities.  The second approach allows to cover the full range of nonlinearities admissible to treatment by Gibbs measure, working instead with a delicate analysis of convergence properties of solutions.  The method of the second approach is quite general, and we shall give applications to the nonlinear Schr\"odinger equation on the unit ball in subsequent works \cite{BB1,BB2}.
\end{abstract}
\maketitle

\section{Introduction}

Our aim in the present work is to study the construction of global solutions for semilinear wave equations set on the unit ball in $\mathbb{R}^3$ with random initial data of supercritical regularity.  In the first part of the paper we prove global well-posedness for the nonlinear wave equation for a certain range of power-type nonlinearities via an application of the method developed in the seminal works \cite{B1,B2,B3}.  In the second part, we treat an expanded range of nonlinearities -- covering  the full range of powers up to the point where measure existence considerations become the dominant obstacle -- by introducing a new approach to obtain arbitrary long time well-posedness of solutions.  In subsequent companion works, we apply the ideas and techniques developed in the second part of the present paper to nonlinear Schr\"odinger equations on the two and three dimensional unit balls to obtain almost sure global well-posedness in these settings \cite{BB1,BB2}.

\subsection{The initial value problem}

We shall consider the Cauchy problem for the nonlinear wave equation posed on the unit ball in $\mathbb{R}^3$ with Dirichlet boundary conditions:
\begin{align*}
(NLW)\quad \left\lbrace\begin{array}{rll}u_{tt}-\Delta u+|u|^\alpha u&=0,&\quad (t,x)\in I\times B,\\
(u,u_t)|_{t=0}&=(u_0,u_1),&\quad x\in B\\
u|_{\partial B}&=0,&\quad t\in I,\end{array}\right.
\end{align*}
where $u:I\times B\rightarrow\mathbb{R}$, $0\in I\subseteq\mathbb{R}$ is a time interval, and $(u_0,u_1)$ is radial data belonging to the support of the Gibbs measure.  

In view of the scaling invariance of the equation, one expects to be able to construct local solutions when the initial data $(u_0,u_1)$ lies in the Sobolev spaces $H_x^s(B)\times H_x^{s-1}(B)$ for the {\it subcritical} and {\it critical} regimes, $s>\frac{3}{2}-\frac{2}{\alpha}$ and $s=\frac{3}{2}-\frac{2}{\alpha}$.  Such results are typically the outcome of fixed-point arguments set in appropriate function spaces.  On the other hand, when the initial data is {\it supercritical} with respect to the canonical scaling, one no longer expects such arguments to apply; indeed, it is well known that in many cases the solution map is not well defined.

In this work, we adopt the strategy of using invariant measures to obtain almost sure global well-posedness for (NLW) beyond the critical regularity threshold.  Through this approach, the ill-posedness of the problem is overcome by requiring that the local and global well-posedness results hold in a probabilistic sense.

More precisely, consider the sequence $(e_n)$ of radial eigenfunctions of $-\Delta$ on $B$ with vanishing boundary conditions, satisfying
\begin{align*}
e_n(x)\sim \frac{\sin(n\pi|x|)}{|x|}
\end{align*}
for every integer $n\geq 1$.  The initial data $(u_0,u_1)$ will be taken as
\begin{align*}
(u_0,u_1)=\bigg(\sum_{n\in\mathbb{N}} \frac{\alpha_n(\omega)}{n\pi}e_n(x),\sum_{n\in\mathbb{N}} \beta_n(\omega)e_n(x)\bigg),
\end{align*}
where $(\alpha_n)$ and $(\beta_n)$ are sequences of independently distributed normalized real-valued Gaussian random variables on a probability space $\Omega$.  As we will see below, this class of data corresponds precisely to the support of the Gibbs measure for (NLW).

In what follows, we shall work with the first-order reformulation of the equation in (NLW) as
\begin{align}
(i\partial_t -\sqrt{-\Delta})u+(\sqrt{-\Delta})^{-1}|\iRe u|^{\alpha}(\iRe u)=0,\label{reformulated}
\end{align}
where solutions $u$ to (NLW) correspond to solutions $\tilde{u}$ of ($\ref{reformulated}$) with initial data $\tilde{u}|_{t=0}=u_0+i(\sqrt{-\Delta})^{-1}u_1$.  

We consider solutions to ($\ref{reformulated}$) in the sense of the {\it Duhamel formula},
\begin{align*}
u(t)=S(t)u(0)-i\int_0^t S(t-\tau)(\sqrt{-\Delta})^{-1}\Big[|\iRe u(\tau)|^{\alpha}\iRe u(\tau)\Big]d\tau,
\end{align*}
where $S(t)$ is the associated linear propagator given by
\begin{align*}
S(t)\bigg(\sum_{n} \alpha_ne_n\bigg)=\sum_{n} \alpha_ne_ne(nt).
\end{align*}

The equation ($\ref{reformulated}$) is of the form $iu_t=(\sqrt{-\Delta})^{-1}(\partial H/\partial \overline{u})$ with conserved Hamiltonian
\begin{align*}
H(\phi)=\int_B |\nabla \phi|^2dx+\frac{1}{\alpha+2}\int_B |\iRe \phi|^{\alpha+2}dx.
\end{align*}

Moreover, introducing the projection operator
\begin{align*}
P_N\bigg(\sum_{n\in \mathbb{N}} u_ne_n\bigg)=\sum_{n\leq N}u_ne_n,
\end{align*}
the Hamiltonian structure of ($\ref{reformulated}$) provides the truncated problem 
\begin{align}
\left\lbrace\begin{array}{rl}(i\partial_t-\sqrt{-\Delta})u+P_N[(\sqrt{-\Delta})^{-1}|\iRe u|^\alpha (\iRe u)]&=0,\\
u|_{t=0}&=P_N\phi,\end{array}\right.\label{truncated}
\end{align}
with the {\it Gibbs measure}
\begin{align}
\mu_G^{(N)}(d\phi)&=e^{-H(P_N\phi)}\prod_{n=1}^N d^2\phi=e^{-\frac{1}{\alpha+2}\int |P_N\textrm{Re}(\phi)|^{\alpha+2}dx}d\mu_F^{(N)},\label{7.3}
\end{align}
where $\mu_F^{(N)}$ is the free measure associated to the Gaussian process 
\begin{align*}
\omega\mapsto P_N\phi^{(\omega)}:=\sum_{n\leq N} \frac{g_n(\omega)}{n\pi}e_n
\end{align*}
with $(g_n(\omega))$ being a sequence of normalized complex-valued Gaussian random variables.

\subsection{Discussion of the main results}

In accordance with the above remarks, our results will concern ($\ref{reformulated}$) equipped with initial data of the form
\begin{align*}
\phi^{(\omega)}(x)=\sum_{n\in\mathbb{N}} \frac{g_n(\omega)}{n\pi}e_n(x).
\end{align*}
\hspace*{\parindent} Standard arguments show that this initial data belongs $\mu_F^{(N)}$-almost surely to the spaces $L_x^p(B)$ for every $p<6$ and $H_x^s(B)$ for every $s<1/2$.  It can be shown that these estimates are sharp, and the data $\phi^{(\omega)}$ therefore belongs almost surely to the supercritical regime for $\alpha\geq 2$.  On the other hand, the range $\alpha<4$ forms exactly the set of powers $\alpha$ for which the Gibbs measures $\mu_G^{(N)}$ and the limiting measure
\begin{align*}
\mu_G=e^{-\frac{1}{\alpha+2}\int |\textrm{Re}(\phi)|^{\alpha+2}dx}d\mu_F
\end{align*}
are well-defined and nontrivial, with $\mu_F$ taken to be the free measure corresponding to the process $\omega\mapsto \phi^{(\omega)}$.

The work contained in the present paper can be divided into two parts.  The first part discusses a contraction-mapping based approach in the spirit of the works \cite{B1,B2,B3} and establishes a global well-posedness result $\mu_F$-almost surely for $\alpha$ in a certain range, namely $\alpha<1+\sqrt{5}$.  In the second part of the paper we introduce a new technique, using a non-contraction mapping based argument, which allows us to treat the full range of powers $\alpha<4$.

\vspace{0.15in}

\noindent {\bf $1.2.1.$ Contraction-mapping based argument.} In Theorem $\ref{thm1}$, we give a treatment of almost sure global well-posedness from the perspective of the contraction-mapping based approach as in the foundational works \cite{B1,B2,B3}.  These arguments take place in the setting of the Fourier restriction spaces $X^{s,b}$ introduced in \cite{B-GAFA1,B-GAFA2}.

\begin{theorem}\label{thm1}
Fix $0<\alpha<1+\sqrt{5}$.  Then, almost surely in $\phi$, for every $0<s<\frac{1}{2}$ and $0<T<\infty$, there exists $u_*\in C_t([0,T);H_x^s(B))$ such that the sequence $(u_N)_{N\geq 1}$ of solutions to (\ref{truncated}) converges to $u_*$ in the space $C_t([0,T);H_x^{s}(B))$.

Moreover, for all $0<s<\frac{5-\alpha}{2}$ and $t\in \mathbb{R}$ the sequence $(u^N(t))$ satisfies
\begin{align}
\sup_N\Vert u^N(t)-e^{it\sqrt{-\Delta}} (P_N\phi)\Vert_{H_x^s}<\infty.\label{eq-thm1}
\end{align}
\end{theorem}

An analogue of this result was established in the case $\alpha<3$ without the use of $X^{s,b}$ spaces in \cite{BT12}.  The essential ingredients which contribute to the proof of Theorem $\ref{thm1}$ are threefold: (1) probabilistic estimates on the initial data and the linear evolution, (2) a Strichartz-type inequality adapted to the $X^{s,b}$ setting which leads to a robust local theory based on the Picard iteration scheme, and (3) long-time bounds on the growth of $H_x^s(B)$ norms of solutions to the truncated equation ($\ref{truncated}$), arising as a result of the invariance of the Gibbs measure, which enable the application of the framework developed in \cite{B1,B2,B3} to conclude the desired convergence result.  %We shall discuss each of these ingredients below.

\vspace{0.15in}

\noindent {\bf $1.2.2.$ Convergence based argument.} Our second result is devoted to the full range $\alpha<4$.  Recall that for $\alpha\geq 4$ the Gibbs measure is no longer well-defined.  As in our discussion of the contraction-mapping based approach, the argument is set in the Fourier restriction spaces $X^{s,b}$.
\begin{theorem}\label{thm2}
Fix $0<\alpha<4$.  Then, almost surely in $\phi$, for every $0<s<\frac{1}{2}$ and $0<T<\infty$, there exists $u_*\in C_t([0,T);H_x^s(B))$ such that the sequence $(u_N)_{N\geq 1}$ of solutions to (\ref{truncated}) converges to $u_*$ in the space $C_t([0,T);H_x^{s}(B))$.  In addition, (\ref{eq-thm1}) remains valid for all $0<s<\frac{5-\alpha}{2}$ and $t\in\mathbb{R}$.
\end{theorem}

Rather than pursuing a contraction-mapping argument as in the proof of Theorem $\ref{thm1}$, the proof of Theorem $\ref{thm2}$ is based on a detailed analysis of the convergence of the sequence of solutions $(u^N)$ to the truncated equations ($\ref{truncated}$).  The first step in performing this analysis is to enhance the probabilistic estimates described in Section $1.2.1$, which hold for the initial data and for the linear evolution, into corresponding estimates for the nonlinear evolution driven by ($\ref{truncated}$); see Lemma $\ref{lem1a}$.

With this refined probabilistic estimate in hand, the next step in the argument is to establish a bootstrap-type inequality for the $X^{s,b}$ norm of $u^{N_1}-u^{N_0}$ for $N_1>N_0\gg 1$.  These bounds are reminiscent of those appearing in the contraction-mapping argument, but allow the presence of an additional error term which converges to zero as $N_0\rightarrow\infty$.  Moreover, in order to close the bootstrap argument it becomes necessary to use short time intervals, with the length of the interval deteriorating as $N_0\rightarrow\infty$.

The last step in the proof of Theorem $\ref{thm2}$ is then to assemble the control given by the boostrap inequality into the desired convergence result globally in time.  To overcome the difficulty that the time interval deteriorates in the limit, we appeal to a covering argument which requires control over the growth of $H_x^s$ norms.  We first establish the weaker property that the sequence $(u^{N_k})$ is $\mu_F$-almost surely converging, with $N_k$ sufficiently rapidly increasing, for instance $N_k=2^k$.  The convergence of the full sequence requires better probabilistic estimates and is obtained by a refinement of the argument.  Finally, since the space $X^{s,b}\subset C_tH_x^s$, the conclusion from Theorem $\ref{thm2}$ follows.

We want to point out that the proof of Theorem $\ref{thm1}$ gives a somewhat stronger result in terms of stability, while Theorem $\ref{thm2}$ makes only the claim of convergence of the solutions of the truncated equations.

\section{Preliminary remarks and fundamental estimates}

Let $B$ denote the unit ball in $\mathbb{R}^3$.  Recall that for $n\geq 1$, the $n$th radial eigenfunction of $-\Delta$ on $B$ with vanishing Dirichlet boundary conditions is given by
\begin{align*}
e_n(x)\sim \frac{\sin(n\pi |x|)}{|x|},
\end{align*}
with associated eigenvalue $n^2$.
Moreover, simple calculations demonstrate
\begin{align}
\lVert e_n\rVert_{L_x^p}\lesssim \left\lbrace\begin{array}{ll}1&p<3,\\(\log n)^{1/3}&p=3,\\n^{1-\frac{3}{p}}&p>3.\end{array}\right.\label{ei_est}
\end{align}

We will often make use of ($\ref{ei_est}$) in the following form: for every sequence $(\alpha_n)\subset \mathbb{C}$ and $p>3$, one has
\begin{align}
\bigg\lVert\Big(\sum_n |\alpha_ne_n(x)|^2\Big)^{1/2}\bigg\rVert_{L_x^p}
&\leq \Big(\sum_n |\alpha_n|^2\lVert e_n\rVert_{L_x^p}^2\Big)^{1/2}
\label{eqb1}\leq \Big(\sum_n |\alpha_n|^2n^{2-\frac{6}{p}}\Big)^{1/2}
\end{align}
and in particular, if $\alpha_n\sim\frac{1}{n}$ then the quantities in ($\ref{eqb1}$) are finite whenever $p<6$.

An essential probabilistic tool in our considerations will be the following estimate for Gaussian processes: if $(g_n)$ is a sequence of IID complex Gaussians with $(\alpha_n)\in \ell^2$ and $2\leq q<\infty$, then 
\begin{align}
\label{gauss-bd-1}\bigg\lVert \sum_{n} \alpha_ng_n(\omega)\bigg\rVert_{L^q(d\omega)}&\lesssim \sqrt{q}\big(\sum_{n} |\alpha_n|^2\big)^{1/2}.
\end{align}

\subsection{$X^{s,b}$ spaces}

Fix $s\in\mathbb{R}$, $\frac{1}{2}<b<1$ and a time interval $I\subset\mathbb{R}$ with $|I|<1/2$.  We define the space $X^{s,b}$ as the collection of functions $f:I\times B\rightarrow\mathbb{R}$ having the representation 
\begin{align}
f(t,x)=\sum_{m,n} \widehat{f}(m,n) e_n(x)e(mt),\quad t\in I,\, x\in B\label{eqa}
\end{align}
with $(\widehat{f}(m,n))\subset \mathbb{C}$, such that the restriction-type norm
\begin{align*}
\lVert f\rVert_{X^{s,b}}&=\inf\bigg(\sum_{m,n} \langle n-m\rangle^{2b}\langle n\rangle^{2s}|\widehat{f}(m,n)|^2\bigg)^{1/2}.
\end{align*}
is finite, where the infimum is taken over all representations of the form ($\ref{eqa}$); see also \cite{B-GAFA1,B-GAFA2}.  Note that the values $\widehat{f}(m,n)$ are not uniquely determined.

\subsection{A deterministic estimate on the nonlinearity}

The following bound of inhomogeneous Strichartz type will form a key tool in treating estimates of the nonlinearity in order to obtain the contraction and bootstrap estimates in each of our approaches.

\begin{lemma}
\label{lem210}
Fix $s\in (0,1)$ and $p>\frac{3}{3-s}$.  Then for every $b>\frac{1}{2}$ sufficiently close to $\frac{1}{2}$ there exists a constant $C>0$ such that 
\begin{align}
\bigg\lVert \int_0^t S(t-\tau)(\sqrt{-\Delta})^{-1}f(\tau)d\tau\bigg\rVert_{X^{s,b}}\leq C\lVert f\rVert_{L_x^p(B;L_t^2(I))}.\label{eqlem210}
\end{align}
\end{lemma}

\begin{proof}
Fix $b>\frac{1}{2}$ to be determined later in the argument.  Then, writing
\begin{align*}
f(t,x)=\sum_{m,n} \widehat{f}(m,n)e_n(x)e(mt),
\end{align*}
we obtain
\begin{align}
\int_0^t S(t-\tau)(\sqrt{-\Delta})^{-1}f(\tau)d\tau&=\sum_{m,n}\frac{\widehat{f}(m,n)}{n}\cdot \frac{e(mt)-e(nt)}{m-n}e_n(x),\label{eq-nonlin-duhamel-nlw}
\end{align}
and therefore 
\begin{align}
\nonumber &\bigg\lVert \int_0^t  S(t-\tau)(\sqrt{-\Delta})^{-1}f(\tau)d\tau\bigg\rVert_{X^{s,b}(I)}\\
%\nonumber &\hspace{0.4in}=\bigg\lVert \sum_{m,n}\int_0^t \frac{e^{i(t-\tau)n}}{n}\widehat{f}(m,n)e_n(x)e(m\tau)d\tau\bigg\rVert_{X^{s,b}(I)}\\
%\nonumber &\hspace{0.4in}\lesssim \bigg\lVert \sum_{n} \bigg(\frac{t\widehat{f}(n,n)}{n}e_n(x)e(nt)+\sum_{m\neq n}\bigg(\frac{\widehat{f}(m,n)}{n(m-n)}\bigg) e_n(x)e(mt)\bigg)\bigg\rVert_{X^{s,b}(I)}\\
%\nonumber &\hspace{1.0in}+\bigg\lVert \sum_n \bigg(\sum_{m\neq n} \frac{\widehat{f}(m,n)}{n(m-n)}\bigg) e_n(x)e(nt)\bigg\rVert_{X^{s,b}(I)}\\
%\nonumber &\hspace{0.4in}\lesssim\bigg(\sum_{m,n}\frac{\langle n-m\rangle^{2b}\langle n\rangle^{2s}|\widehat{f}(m,n)|^2}{|n|^2|m-n|^2}\bigg)^{1/2}\\
%&\hspace{1.0in}+\bigg(\sum_n \langle n\rangle^{2s}\big|\sum_{m\neq n} \frac{\widehat{f}(m,n)}{n(m-n)}\big|^2\bigg)^{1/2}.\label{eq_aaa_2}
\nonumber &\hspace{0.4in}\lesssim \bigg(\sum_{m,n} \frac{|\widehat{f}(m,n)|^2}{\langle n\rangle^{2(1-s)}\langle m-n\rangle^{2(1-b)}}\bigg)^{1/2}\\
\nonumber &\hspace{1.8in}+\bigg(\sum_n \bigg(\sum_{m\neq n} \frac{|\widehat{f}(m,n)|}{\langle n\rangle^{2(1-s)}\langle m-n\rangle}\bigg)^2\bigg)^{1/2}\\
\label{eq_aaa_2b}&\hspace{0.4in}\lesssim \bigg(\sum_{m,n} \frac{|\widehat{f}(m,n)|^2}{\langle n\rangle^{2(1-s)}\langle m-n\rangle^{2(1-b)}}\bigg)^{1/2},
\end{align}
where the last line follows from the Cauchy-Schwarz inequality and the hypothesis $b>1/2$.

To estimate ($\ref{eq_aaa_2b}$) we argue by duality, which gives
\begin{align*}
(\ref{eq_aaa_2b})&\leq \sup \int_0^1\int_{B_3} f(t,x)\overline{g(t,x)}dxdt,
\end{align*}
where the supremum is over functions $g$ of the form
\begin{align}
g(t,x)=\sum_{m,n} \frac{g_{m,n}}{\langle n\rangle^{1-s}\langle m-n\rangle^{1-b}}e_n(x)e(mt)\quad\textrm{with}\quad \sum_{m,n} |g_{m,n}|^2\leq 1.\label{eq_aaa_2d}
\end{align}

Fix such a function $g$.  Invoking the Cauchy-Schwarz inequality in time, we have
\begin{align}
\int_0^t \int_{B_3} f(t,x)\overline{g(t,x)}dxdt&\leq \int_{B_3} \lVert f(t,x)\rVert_{L_t^2}\lVert g(t,x)\rVert_{L_t^2}dx.\label{eq12a}
\end{align}
On the other hand, H\"older's inequality gives
\begin{align*}
\lVert g(t,x)\rVert_{L_t^2}&\leq \bigg(\sum_{m}\big(\sum_n \frac{|g_{m,n}|}{\langle n\rangle^{1-s}\langle m-n\rangle^{1-b}}|e_n(x)|\big)^{2}\bigg)^{1/2}\\
&\leq \bigg(\sum_{m,n}\frac{\alpha_m^2}{\langle n\rangle^{2(1-s)}\langle m-n\rangle^{2(1-b)}}|e_n(x)|^2\bigg)^{1/2}\\
&=\bigg(\sum_{n}\frac{\beta_n}{\langle n\rangle^{2(1-s)-\epsilon}}|e_n(x)|^2\bigg)^{1/2},
\end{align*}
where we have fixed $\epsilon>0$ and set 
\begin{align*}
\alpha_m^2:=\sum_n |g_{m,n}|^2\quad\textrm{and}\quad \beta_n:=\sum_{m} \frac{\alpha_m^2}{\langle n\rangle^{\epsilon}\langle m-n\rangle^{2(1-b)}}.
\end{align*}

Now, using H\"older's inequality in space and the bound $\sum_n\beta_n<\infty$ for $b$ sufficiently close to $1/2$ (as a consequence of $\sum_{m,n} |g_{m,n}|^2\leq 1$), followed by ($\ref{ei_est}$), we obtain
\begin{align}
\nonumber (\ref{eq12a})&\leq \lVert f\rVert_{L_x^{p}L_t^2}\bigg\lVert \sum_n \frac{\beta_n}{\langle n\rangle^{2(1-s)-\epsilon}}|e_n(x)|^2\bigg\rVert_{L_x^{q/2}}^{1/2}\\
\nonumber &\leq \lVert f\rVert_{L_x^{p}L_t^2}\sup_n \frac{\lVert  e_n(x)\rVert_{L_x^{q}}}{\langle n\rangle^{(1-s)-\frac{\epsilon}{2}}}\left(\sum_n\beta_n\right)^{1/2}\\
&\lesssim \lVert f\rVert_{L_x^{p}L_t^2}\label{eq-star-1}
\end{align}
with $\frac{1}{p}+\frac{1}{q}=1$ and $\epsilon$ sufficiently small, provided that $3<q<\frac{3}{s}$, which corresponds to the condition $p>\frac{3}{3-s}$.

Combining ($\ref{eq_aaa_2b}$) with ($\ref{eq-star-1}$) completes the desired estimate.
\end{proof}

\section{Proof of Theorem $\ref{thm1}$: contraction-mapping based approach}

In this section, we prove Theorem $\ref{thm1}$, which establishes almost sure global well-posedness of solutions to the reformulated nonlinear wave equation ($\ref{reformulated}$) for the range $1\leq \alpha<1+\sqrt{5}$ by using a fixed-point argument.  

We first consider the case $\alpha=3$, for which the equation ($\ref{reformulated}$) becomes 
\begin{align}
(i\partial_t-\sqrt{-\Delta})u+(\sqrt{-\Delta})^{-1}|\iRe u|^3(\iRe u)=0,\label{1.0}
\end{align}
with Duhamel formula 
\begin{align}
u(t)=S(t)u_0-i\int_0^t S(t-\tau)(\sqrt{-\Delta})^{-1}(|\iRe u(\tau)|^3\iRe u(\tau))d\tau.\label{1.1}
\end{align}

\subsection{Estimate on the nonlinear term in (\ref{1.1})}
The main issue in establishing the fixed-point iteration for ($\ref{1.0}$) is to obtain suitable estimates on the nonlinear term.  This is accomplished by appealing to Lemma $\ref{lem210}$ and estimating the resulting norm.

In particular, suppose that $\lVert u\rVert_{X^{s,b}}\leq 1$ and apply ($\ref{eqlem210}$) of Lemma $\ref{lem210}$ with $f=|\iRe u|^3\iRe u$.  Note first that $u$ may be written as
\begin{align}
u(x,t)=\sum_{m,n} \langle n\rangle^{-s}\langle m-n\rangle^{-b}a_{n,m}e_n(x)e(mt)\label{3.1}
\end{align}
for some $(a_{n,m})$ satisfying $\sum_{m,n} |a_{n,m}|^2\leq 1$.  

Setting  
\begin{align*}
\beta_\ell^2=\sum_n |a_{n,n+\ell}|^2,
\end{align*}
we rewrite ($\ref{3.1}$) as
\begin{align}
\sum_{\ell\in\mathbb{Z}} \langle\ell\rangle^{-b}\beta_\ell\bigg\{\sum_n \langle n\rangle^{-s}\beta_\ell^{-1}a_{n,n+\ell}e_n(x)e(nt)\bigg\}e(\ell t)\label{3.2}
\end{align}

Note that by the Cauchy-Schwarz inequality and the condition $b>1/2$ we have
\begin{align*}
\sum_{\ell\in \mathbb{Z}} \langle \ell\rangle^{-b}\beta_\ell \leq \bigg(\sum_{\ell\in\mathbb{Z}}\langle \ell\rangle^{-2b}\bigg)^{1/2}\bigg(\sum_{\ell\in\mathbb{Z}} \beta_\ell^2\bigg)^{1/2}\lesssim 1.
\end{align*}
It follows from convexity that we may replace $u$ by a function of the form
\begin{align}
u_1(t,x)=\sum_n \langle n\rangle^{-s}a_ne_n(x)e(nt)\label{3.3}
\end{align}
with $\sum_n |a_n|^2\leq 1$; indeed, in view of Lemma $\ref{lem210}$ the relevant bound is 
\begin{align*}
\lVert u\rVert_{L_x^{4p}L_t^8}&\leq \sum_{\ell\in \mathbb{Z}} \langle \ell\rangle^{-b}\beta_{\ell}\bigg\lVert \sum_n \langle n\rangle^{-s}\beta_{\ell}^{-1}a_{n,n+\ell}e_n(x)e(nt)\bigg\rVert_{L_x^{4p}L_t^{8}}\\
&\leq C\sup_{\ell\in\mathbb{N}} \bigg\lVert\sum_n \langle n\rangle^{-s}\beta_{\ell}^{-1}a_{n,n+\ell}e_n(x)e(nt)\bigg\rVert_{L_x^{4p}L_t^{8}}.
\end{align*}
It then remains to estimate
\begin{align*}
\lVert u_1\rVert_{L_x^{4p}L_t^{8}}&\lesssim T^{1/\delta}\lVert u_1\rVert_{L_x^{4p}H_t^{\frac{3}{8}+\delta}}\\
&\leq T^{1/\delta}\bigg\lVert \bigg(\sum_n \langle n\rangle^{-2s+\frac{3}{4}+2\delta}|a_n|^2e_n(x)^2\bigg)^{1/2}\bigg\rVert_{L_x^{4p}}\\
&\leq T^{1/\delta}\bigg(\sum_n\langle n\rangle^{-2s+\frac{3}{4}+2\delta}|a_n|^2\lVert e_n\rVert_{L_x^{4p}}^2\bigg)^{1/2}\\
&\leq T^{1/\delta}\max_n \frac{\lVert e_n\rVert_{L_x^{4p}}}{\langle n\rangle^{s-\frac{3}{8}-\delta}}\\
&\leq CT^{1/\delta}
\end{align*}
where we have fixed $\delta>0$ sufficiently small, provided that
\begin{align}
1-\frac{3}{4p}<s-\frac{3}{8}.\label{3.4}
\end{align}
Recalling the condition $p>\frac{3}{3-s}$, we require
\begin{align}
s>\frac{5}{6}.\label{3.5}
\end{align}

\subsection{Random data and local well-posedness for (\ref{1.0})} In order to proceed with our treatment of the local well-posedness theory for ($\ref{1.0}$), we now discuss
the relevant probabilistic aspects.

Take
\begin{align}
u_0(x)=u_{0,\omega}(x)=\sum_n \frac{g_n(\omega)}{n\pi}e_n(x)\label{4.1}
\end{align}
according to the support of the Gibbs measure,
and 
fix $s<1$ and $b>1/2$, respectively close to $1$ and $1/2$.  Our aim is to prove that for $T$ sufficiently small, except
for an $\omega$-set of small measure, the solution $u$ of ($\ref{1.1}$) satisfies 
\begin{align*}
u\in S(t)\phi+B
\end{align*}
with $B$ a small ball in the space $X^{s,b}$ (with the qualifications that are usual in this context; c.f. \cite{B1}).

Writing 
\begin{align*}
u(t)=S(t)\phi+v(t)=:v_0(t)+v(t)
\end{align*}
with $\lVert v\rVert_{X^{s,b}}<1$, we first verify that the second term in ($\ref{1.1}$) is indeed bounded in $X^{s,b}$.  

Note that the function $v$ satisfies the integral equation 
\begin{align}
v(t)=\int_0^t S(t-\tau)(\sqrt{-\Delta})^{-1}|\iRe (v_0(\tau)+v(\tau))|^3\iRe(v_0(\tau)+v(\tau))d\tau.\label{eq_aaa1}
\end{align}
Applying Lemma $\ref{lem210}$ with $f=|\iRe u|^3 \iRe u$, it suffices to bound
\begin{align*}
\lVert f\rVert_{L_x^pL_t^2}\leq \lVert u\rVert_{L_x^{4p}L_t^8}^4
\end{align*}
and hence
\begin{align}
\lVert u\rVert_{L_x^{4p}L_t^8}\leq \lVert v_0\rVert_{L_x^{4p}L_t^8}+\lVert v\rVert_{L_x^{4p}L_t^8}\label{4.3}
\end{align}
where the second term of ($\ref{4.3}$) was already estimated in Section $3.1$, assuming $s>\frac{5}{6}$.

To estimate the first term in ($\ref{4.3}$), we give a probablistic lemma for the linear evolution:
\begin{lemma}
\label{prop_aaa_1}
Fix $0<T<+\infty$ and let $2\leq p<6$ be given together with $2\leq q<\infty$.  Then there exists $c>0$ such that for every $\lambda>0$ and $N\in\mathbb{N}$, one has
\begin{align*}
\mu_F^{(N)}\bigg(\bigg\{\phi_N:\lVert S(t)\phi_N\rVert_{L_x^pL_t^q}>\lambda\bigg\}\bigg)\lesssim \exp(-c\lambda^2).
\end{align*}
\end{lemma}
\begin{proof}
Fix $r\geq\max\{p,q\}$.  The Tchebyshev and Minkowski inequalities then give
\begin{align}
\nonumber &\mu_F^{(N)}(\{\phi_N:\lVert S(t)\phi_N\rVert_{L_x^pL_t^q}>\lambda\})\\
\nonumber &\hspace{1.2in}\leq \frac{1}{\lambda^r}\mathbb{E}_{\mu_F}\bigg[\lVert S(t)\phi\rVert^r_{L_x^pL_t^q}\bigg]\\
\nonumber &\hspace{1.2in}\leq \frac{1}{\lambda^r}\int \bigg\lVert \sum_n \frac{g_n(\omega)}{n}e_n(x)e(nt)\bigg\rVert_{L_x^pL_t^q}^rd\mu_F^{(N)}(\phi)\\
\nonumber &\hspace{1.2in}\lesssim \bigg(\frac{\sqrt{r}}{\lambda}\bigg)^{r}\bigg\lVert \bigg(\sum_n \frac{|e_n(x)|^2}{n^{2}}\bigg)^{1/2}\bigg\rVert_{L_x^p}^r\\
\nonumber &\hspace{1.2in}\lesssim \bigg(\frac{\sqrt{r}}{\lambda}\bigg)^{r}
\end{align}
where we have again used ($\ref{eqb1}$) and ($\ref{gauss-bd-1}$) together with the hypothesis $p<6$.  The desired conclusion now follows by optimizing in the choice of $r$.
\end{proof}

We will also need
\begin{lemma}
\label{lem-prob-data}
Fix $s\in [0,1/2)$, $0<T<+\infty$, and let $2\leq p<\frac{6}{1+2s}$ be given.  Then there exists $c>0$ such that for every $\lambda>0$ and $N\in\mathbb{N}$, one has
\begin{align*}
\mu_F^{(N)}\bigg(\bigg\{\phi_N:\lVert (\sqrt{-\Delta})^s\phi_N\rVert_{L_x^p}>\lambda\bigg\}\bigg)\lesssim \exp(-c\lambda^2).
\end{align*}
\end{lemma}
\begin{proof}
Without loss of generality assume $p>3$.  Let $\lambda>0$ and fix $p_1=p_1(\lambda)>p$.  The Tchebyshev and Minkowski inequalities give
\begin{align}
\nonumber &\mu_F^{(N)}(\{\phi_N:\lVert (\sqrt{-\Delta})^s\phi_N\rVert_{L_x^p}>\lambda\})\\
\nonumber &\hspace{1.2in}\leq \frac{1}{\lambda^{p_1}}\mathbb{E}_{\mu_F^{(N)}}\bigg[\lVert (\sqrt{-\Delta})^s\phi_N\rVert_{L_x^p}^{p_1}\bigg]\\
\nonumber &\hspace{1.2in}\leq \frac{1}{\lambda^{p_1}}\bigg\lVert \bigg(\int |(\sqrt{-\Delta})^s\phi_N(x)|^{p_1}d\mu_F^{(N)}(\phi_N)\bigg)^{1/p_1}\bigg\rVert_{L_x^p}^{p_1},\\
\nonumber &\hspace{1.2in}\leq \bigg(\frac{\sqrt{p_1}}{\lambda}\bigg)^{p_1}\bigg\lVert\bigg(\sum_n \frac{|e_n(x)|^2}{n^{2(1-s)}}\bigg)^{1/2}\bigg\rVert_{L_x^p}^{p_1}\\
\nonumber &\hspace{1.2in}\lesssim \bigg(\frac{\sqrt{p_1}}{\lambda}\bigg)^{p_1},
\end{align}
where we have used ($\ref{eqb1}$) and ($\ref{gauss-bd-1}$) along with the hypothesis $p<\frac{6}{1+2s}$.  Optimizing in the choice of $p_1$ gives the desired bound.
\end{proof}

\vspace{0.1in}

Elaborating the above observations shows that, if we fix any $\frac{5}{6}<s<1$, the initial value problem 
\begin{align}
(\ref{1.0})\quad\textrm{with}\quad u(0)=\phi_\omega,\quad t\in [0,T]\label{5.1}
\end{align}
has a unique solution $u=u(x,t)$ satisfying
\begin{align}
\lVert u-S(t)\phi_\omega\rVert_{X^{s,b}}<1\label{5.2}
\end{align}
provided $\omega$ is restricted to the complement of a set $\Omega(T)$, where
\begin{align}
|\Omega(T)|<e^{-(1/T)^c}\quad\textrm{as}\quad T\rightarrow 0\label{5.3}
\end{align}
for some $c>0$.  
More precisely, 
\begin{proposition}
\label{prop-fp-1}
Given $A\gg 1$, there is $T=A^{-c}$ such that on the time interval $[t_0,t_0+T]$, $t_0<O(1)$, 
the solution $u$ to ($\ref{5.1}$) exists for all $\phi$ outside a singular set of $\mu_F$-measure at most $O(\exp(-cA^c))$ for some constant $c>0$, and satisfies, for $0<s<1$, $t\in [t_0,t_0+T)$,
\begin{align}
\lVert u(t)-S(t)\phi\rVert_{X^{s,b}}\lesssim A^{4}\label{eq-claim-1}
\end{align}
for every $b>1/2$ sufficiently close to $1/2$.  Hence,
\begin{align}
\lVert u(t)-S(t)\phi\rVert_{H_x^s}\lesssim A^4\quad\textrm{for}\quad t\in [t_0,t_0+T].\label{3.14p}
\end{align}
Moreover, for $0<s'<\frac{1}{2}$,
\begin{align}
\lVert u(t)\rVert_{H_x^{s'}}\lesssim A^4.\label{3.15}
\end{align}
\end{proposition}

\begin{proof}
For each $v\in X^{s,b}([0,T))$, define 
\begin{align*}
\Phi(v)=\int_0^t S(t-\tau)(\sqrt{-\Delta})^{-1}|\iRe (v_0(\tau)+v(\tau))|^3\iRe(v_0(\tau)+v(\tau))d\tau.
\end{align*}

Fix $0<s<1$ and $b>1/2$ (with $b$ close to $1/2$) to be determined later in the argument, along with $R>0$.  Now, fix a parameter $p\in (\frac{3}{3-s},\frac{3}{2})$, and let $B_{R}$ be the ball of radius $R$ in $X^{s,b}([0,T))$.  Invoking Lemma \ref{lem210}, we then obtain the estimate
\begin{align}
\nonumber \lVert \Phi(v)\rVert_{X^{s,b}}&\lesssim \lVert v_0\rVert_{L_x^{4p}L_t^{8}}^{4}+\lVert v\rVert_{L_x^{4p}L_t^{8}}^{4}\\
&\lesssim A^{4}+\lVert v\rVert_{L_x^{4p}L_t^{8}}^{4}.\label{eq--aab1}
\end{align}
provided that $\phi$ belongs to the set 
\begin{align*}
\Omega_{A  }=\{\phi:\lVert S(t)\phi\rVert_{L_x^{4p}L_t^{8}}\leq A  \}.
\end{align*}

At this point, we pause to note that the condition $p<\frac{3}{2}$ implies that we may apply Lemma $\ref{prop_aaa_1}$ to obtain 
\begin{align*}
\mu_F(\Omega\setminus \Omega_A)=O(\exp(-cA^c))
\end{align*}
for some constants $c,C>0$.  This condition on $p$ is compatible with the condition $p>\frac{3}{3-s}$ for $s<1$.

From the considerations in Section 3.1, we have
\begin{align}
\lVert v\rVert_{L_x^{4p}L_t^{8}}\leq T^\delta R.\label{eq_aaa_3}
\end{align}

Combining ($\ref{eq--aab1}$) and ($\ref{eq_aaa_3}$), we have shown
\begin{align*}
\lVert \Phi(v)\rVert_{X^{s,b}([0,T])}\leq C[A^{4}+(T^\delta R)^{4}].
\end{align*}
Choosing $R=2CA^{4}$ and $T<R^{-\frac{1}{\delta}}$, we obtain 
\begin{align*}
\lVert \Phi(v)\rVert_{X^{s,b}([0,T))}\leq R
\end{align*}
so that $\Phi$ maps the ball $B_R$ to itself.  Repeating the above estimates combined with the inequality
\begin{align*}
\big| |\iRe f|^{3} f-|\iRe g|^{3} g\big|&\leq (|\iRe f|^{3}+|\iRe g|^{3})|f-g|,
\end{align*}
we obtain that $\Phi$ is a contraction for $T>0$ sufficiently small depending on $A$; this gives the desired local existence result.

The claim ($\ref{eq-claim-1}$) then follows directly from $v\in B_R$, while ($\ref{3.15}$) is a consequence of the embedding $X^{s,b}\subset L_t^\infty H_x^s$ for $b>1/2$ and Lemma $\ref{lem-prob-data}$.
\end{proof}

Using Proposition $\ref{prop-fp-1}$, the same scheme as used in \cite{B1}, and which exploits essentially the invariance of the Gibbs measure under the flow, permits then to obtain solutions global in time, with
$u-S(t)\phi_\omega\in \cap_{s<1}X^{s,b}$, almost surely in $\omega$.

The key ingredient in this stage of the argument is to obtain suitable bounds on the growth of the norms $H_x^s$, $s<1/2$.  

\begin{proposition}
\label{prop-fp-2}
Fix $0<s<\frac{1}{2}$ and $\delta>0$.  Then there exists a set $\Sigma=\Sigma(\delta)\subset\Omega$ such that $\mu_F^{(N)}(\{P_N\phi:\omega\in \Sigma\})>1-\delta$, and for $\omega\in \Sigma$ the solution $u=u_N$ of (\ref{truncated}) (with $\alpha=3$) satisfies
\begin{align*}
\lVert u(t)\rVert_{L_t^\infty([0,T];H_x^s(B))}\leq C\bigg(\log \frac{T}{\delta}\bigg)^{C}.
\end{align*}
for every $0<T<+\infty$.  Here $C$ is some constant.
\end{proposition}

\begin{proof}
The argument follows a familiar approach, and is based on the invariance of the Gibbs measure; see \cite[Lemma $3.25$]{B1}.  Fix $T\gg 1$, $\delta\ll 1$ and set $A=C\left(\log\frac{T}{\delta}\right)^C$ for an appropriate constant $C$.  Let $\Sigma_A$ be the good data set given by Proposition $\ref{prop-fp-1}$.

Let $V$ be the operator $V_N(\tau)$, where $V_N(t)$ denotes the (nonlinear) evolution operator corresponding to the truncated equation ($\ref{truncated}$), and $\tau=A^{-C}$ the time interval given by Proposition $\ref{prop-fp-1}$.  Define
\begin{align*}
\Sigma':=\Sigma_A\cap V^{-1}\Sigma_A\cap V^{-2}\Sigma_A\cap \cdots \cap V^{-\lfloor \frac{T}{\tau}\rfloor}\Sigma_A.
\end{align*}
Invoking the invariance of $\mu_G^{(N)}$ with respect to the flow $V_N(t)$ we therefore have
\begin{align}
\mu_G^{(N)}((\Sigma')^c)&\leq \frac{T}{\tau}\cdot\mu_G^{(N)}((\Sigma_A)^c)\leq \frac{T}{\tau}\exp(-cA^c)<\delta\label{3.18}
\end{align}
by our choice of $A$.

It is clear from the construction that ($\ref{3.14p}$), ($\ref{3.15}$) hold on the entire time interval $[0,T]$ for $\phi\in\Sigma'$.  The proof of Proposition $\ref{prop-fp-2}$ is then easily completed by intersecting the sets $\Sigma'=\Sigma'_{T,\delta}$ over suitable sequences $T\rightarrow\infty$, $\delta\rightarrow 0$.
\end{proof}

The proof of Theorem $\ref{thm1}$ in the case $\alpha=3$ is completed by standard approximation arguments, see for instance \cite{B1}.  The long-time bounds on the truncated 
evolution ($\ref{truncated}$) given by Proposition $\ref{prop-fp-2}$ allow us to iteratively invoke the local-theory of Proposition $\ref{prop-fp-1}$ for the original evolution 
($\ref{1.0}$) to extend the local well-posedness given by to arbitrarily long  time intervals.  This completes the proof of Theorem $\ref{thm1}$ in this case.

\subsection{Extension to the range $\alpha<1+\sqrt{5}$ for the equation ($\ref{reformulated}$)}

We examine the limitation on powers $\alpha$ for which the conclusion from Section 3.2 remains valid.  Let
\begin{align}
u\in S(t)u_0+B,\quad u=v_0+v\label{6.2}
\end{align}
with $B$ a small ball in $X^{s,b}$, $s<1$ to be specified.  Applying again Lemma $\ref{lem210}$, we need to estimate
\begin{align}
\lVert v_0\rVert_{L_x^{\beta p}L_t^{2\beta}}+\lVert v\rVert_{L_x^{\beta p}L_t^{2\beta}}\label{6.3}
\end{align}
where $\beta=1+\alpha$ and $p>\frac{3}{3-s}$.

The same calculation as in the previous section shows that
\begin{align}
\lVert v_0\rVert_{L_x^{\beta p}L_t^{2\beta}}\lesssim \bigg(\sum_n \frac{\lVert e_n\rVert_{L_x^{\beta p}}^2}{\langle n\rangle^2}\bigg)^{1/2}\lesssim \bigg(\sum_n \bigg(\frac{1}{\langle n\rangle}\bigg)^{\frac{6}{\beta p}}\bigg)^{1/2}\label{6.4}
\end{align}
which leads to the condition
\begin{align}
s<\frac{6-\beta}{2}=\frac{5-\alpha}{2}.\label{6.5}
\end{align}

From the discussion of Section 3.2, $v$ may be replaced by a function $u_1$ of the form ($\ref{3.3}$) and
\begin{align}
\nonumber \lVert u_1\rVert_{L_x^{\beta p}L_t^{2\beta}}&\leq \lVert u_1\rVert_{L_x^{\beta p}H_t^{\frac{1}{2}(1-\frac{1}{\beta})}}\\
\nonumber &\leq \bigg\lVert \bigg(\sum_n \langle n\rangle^{-2s+1-\frac{1}{\beta}}|a_n|^2e_n(x)^2\bigg)^{1/2}\bigg\rVert_{L_x^{\beta p}}\\
&\leq \max_n \frac{\lVert e_n\rVert_{L_x^{\beta p}}}{\langle n\rangle^{s-\frac{1}{2}+\frac{1}{2\beta}}}\label{6.6}
\end{align}

Again from ($\ref{ei_est}$) this leads to the condition
\begin{align*}
1-\frac{3}{\beta p}<s-\frac{1}{2}+\frac{1}{2\beta}
\end{align*}
or, equivalently
\begin{align}
s>\frac{3\beta-7}{2(\beta-1)}=\frac{3\alpha-4}{2\alpha}\label{6.7}
\end{align}

In view of ($\ref{6.5}$), ($\ref{6.7}$), we see that the conclusion from Section 3.2 holds for the equation ($\ref{reformulated}$) with $s<\frac{5-\alpha}{2}$, provided
\begin{align*}
\alpha^2-2\alpha-4<0,\quad\textrm{that is},\quad \alpha<1+\sqrt{5}.
\end{align*}
This completes the proof of Theorem $\ref{thm1}$.

\section{Proof of Theorem $\ref{thm2}$: convergence based approach}

In this section we prove Theorem $\ref{thm2}$, which establishes global well-posedness for the nonlinear wave equation reformulated as ($\ref{reformulated}$), $\mu_F$-almost surely in the statistical ensemble.  In particular, for arbitrary $\alpha<4$, which is the range for which the problem admits a well-defined Gibbs measure, one may still produce for almost all data a solution of ($\ref{reformulated}$) which is obtained as a limit of the solutions of the truncated equations and satisfying
\begin{align}
u(t)-S(t)u(0)\in \bigcap_{s<\frac{5-\alpha}{2}}X^{s,\frac{1}{2}+}.\label{7.1}
\end{align}

The argument is based on estimates on the solutions of the truncated equations rather than constructing a solution of ($\ref{reformulated}$) by Picard iteration.

Firstly, note that the ODE ($\ref{truncated}$) has a unique solution, which is global in time (see \cite[Prop. $2.2$]{BT12} (note that we already use here the fact that the equation is defocusing, since we rely on the {\it a priori} bound given by the Hamiltonian).  Moreover, recall that the flow map
\begin{align*}
\phi_N=P_N\phi\mapsto u_N(t)
\end{align*}
leaves the Gibbs measure $\mu_G^{(N)}$ invariant.

The next probabilistic estimate is crucial in our argument.
\begin{lemma}
\label{lem1a}
Fix a time interval $[0,T]$ and let $1\leq p<6$ and $1\leq q<\infty$.  Then 
\begin{align}
\mathbb{E}_{\mu_F}\Big[\lVert u_\phi\rVert_{L_x^pL_t^q}\Big]<C(p,q,T)\label{7.6}
\end{align}
denoting $u_\phi=u_N$, $u_N(0)=\phi_N$.  In fact, there is the stronger distributional inequality
\begin{align}
\label{seventeen}\mu_F^{(N)}\bigg(\bigg\{\phi_N:\lVert u_N\rVert_{L_x^p(B;L_t^q([0,T))}>\lambda\bigg\}\bigg)\lesssim e^{-c(\lambda T^{-1/q})^c}.
\end{align}
for some $c>0$
\end{lemma}

\begin{proof}
Without loss of generality, suppose $3<p<q$.  We first observe that it suffices to prove ($\ref{seventeen}$) with $\mu_F^{(N)}$ replaced by the projected Gibbs measure $\mu_G^{(N)}$.  Indeed, from ($\ref{7.3}$), setting
\begin{align*}
E_{\lambda,N}:=\bigg\{\phi_N:\lVert u_N\rVert_{L_x^p(B;L_t^q([0,T))}>\lambda\bigg\}
\end{align*}
and taking $\lambda_1>0$ arbitrary,
\begin{align}
\nonumber \mu_F^{(N)}(E_{\lambda,N})&\leq \mu_F^{(N)}(E_{\lambda,N}\cap \{\phi_N:\lVert \phi_N\rVert_{L_x^{\alpha+2}}\leq\lambda_1\})\\
\nonumber &\hspace{0.8in}+\mu_F^{(N)}(\{\phi_N:\lVert \phi_N\rVert_{L_x^{\alpha+2}}>\lambda_1\})\\
&\leq e^{\frac{1}{\alpha+2}\lambda_1^{\alpha+2}}\mu_G^{(N)}(E_{\lambda,N})+\mu_F^{(N)}(\{\phi_N:\lVert \phi_N\rVert_{L_x^{\alpha+2}}>\lambda_1\})\label{eqb2}
\end{align}
where to obtain the second inequality we have used the Tchebyshev inequality via
\begin{align}
\nonumber &\mu_F^{(N)}(E_{\lambda,N}\cap \{\phi_N:\lVert \phi_N\rVert_{L_x^{\alpha+2}}\leq\lambda_1\})\\
\nonumber &\hspace{1.2in}\leq e^{\frac{1}{\alpha+2}\lambda_1^{\alpha+2}}\int e^{-\frac{1}{\alpha+2}\int|\phi_N|^{\alpha+2}dx}\chi_{E_{\lambda,N}}d\mu_F^{(N)}(\phi_N)\\
\nonumber &\hspace{1.2in}\leq e^{\frac{1}{\alpha+2}\lambda_1^{\alpha+2}}\mu_G^{(N)}(E_{\lambda,N}).
\end{align}

To estimate the second term in $(\ref{eqb2})$ note that, for each $q_2>\alpha+2$, Tchebyshev's inequality followed by Fubini's theorem and the Gaussian bound $(\ref{gauss-bd-1})$ give
\begin{align}
\nonumber \mu_F^{(N)}(\{\phi_N:\lVert\phi_N\rVert_{L_x^{\alpha+2}}>\lambda_1\})&\lesssim \frac{1}{(\lambda_1)^{q_2}}\bigg\lVert \sum_{n=1}^N \frac{g_n(\omega)}{n}e_n(x)\bigg\rVert_{L_\omega^{q_2}L_x^{\alpha+2}}^{q_2}\\
\nonumber &\lesssim \bigg(\frac{\sqrt{q_2}}{\lambda_1}\bigg)^{q_2}\bigg\lVert\bigg(\sum_{n=1}^N \frac{|e_n(x)|^2}{n^2}\bigg)^{\frac{1}{2}}\bigg\rVert_{L_x^{\alpha+2}}^{q_2}\\
\nonumber &\lesssim \bigg(\frac{\sqrt{q_2}}{\lambda_1}\bigg)^{q_2}\bigg(\sum_{n=1}^N n^{-\frac{6}{\alpha+2}}\bigg)^{\frac{q_2}{2}}.
\end{align}
Keeping $\lambda_1>0$ fixed and optimizing in $q_2$ then gives
\begin{align}
\mu_F^{(N)}(\{\phi_N:\lVert \phi_N\rVert_{L_x^{\alpha+2}}>\lambda_1\})\lesssim e^{-c\lambda_1^2}.\label{eq19}
\end{align}

Thus, if we establish an estimate of the form
\begin{align}
\mu_G^{(N)}(\{\phi:\lVert u_\phi\rVert_{L_x^pL_t^q}>\lambda\})\lesssim \exp(-\lambda^{c_1})\label{7.9}
\end{align}
this will imply
\begin{align*}
\mu_F(\{\phi:\lVert u_\phi\rVert_{L_x^pL_t^q}>\lambda\})&\leq \min_{\lambda_1} \bigg(e^{\frac{1}{\alpha+2}\lambda_1^{\alpha+2}-\lambda^{c_1}}+e^{-c\lambda_1^2}\bigg)\lesssim e^{-\lambda^{c'}}.
\end{align*}

We now establish ($\ref{7.9}$).  Fix $q_1=q_1(\lambda)>q$ to be determined later in the argument and note that
\begin{align}
\mathbb{E}_{\mu_G}^{(N)}\bigg[\bigg(\lVert u^N\rVert_{L^p_xL_t^q}\bigg)^{q_1}\bigg]^{1/q_1}&\leq \bigg\lVert \bigg(\int |u^N(t,x)|^{q_1}d\mu_G^{(N)}(\phi)\bigg)^{1/q_1}\bigg\rVert_{L_x^pL_t^q}.\label{eq_aaa2}
\end{align}
On the other hand, the invariance of $\mu_G^{(N)}$ under the evolution given by ($\ref{truncated}$) guarantees
\begin{align*}
\int |u^N(t,x)|^{q_1}d\mu^{(N)}_G(\phi)&=\int |\phi_N(x)|^{q_1}d\mu^{(N)}_G(\phi)\\
&\leq \int |\phi_N(x)|^{q_1}d\mu_F^{(N)}(\phi)\\
&\leq C^{q_1}\int \bigg|\sum_{n=1}^N \frac{g_n(\omega)}{n}e_n(x)\bigg|^{q_1}d\omega\\
&\leq (C\sqrt{q_1})^{q_1}\left(\sum_{n=1}^N \frac{|e_n(x)|^2}{n^2}\right)^{q_1/2}.
\end{align*}
where the last inequality results from standard Gaussian bounds as in the proof of Lemma $\ref{prop_aaa_1}$.   

Inserting this bound into ($\ref{eq_aaa2}$), we obtain the estimate
\begin{align*}
\mathbb{E}_{\mu_G}^{(N)}\bigg[\bigg(\lVert u^N\rVert_{L^p_xL_t^q}\bigg)^{q_1}\bigg]^{1/q_1}&\leq CT^{1/q}\sqrt{q_1}\bigg\lVert \left(\sum_{n=1}^N \frac{|e_n(x)|^2}{n^2}\right)^{1/2}\bigg\rVert_{L_x^p}\\
&\leq CT^{1/q}\sqrt{q_1}\bigg(\sum_{n=1}^N n^{-6/p}\bigg)^{1/2},
\end{align*}
where we have used the estimate ($\ref{eqb1}$) with $\alpha_n=\frac{1}{n}$.  

Recalling the hypothesis $p<6$ and applying Tchebyshev's inequality, we then obtain
\begin{align*}
\mu_G^{(N)}(E_{\lambda,N})&\lesssim \frac{1}{\lambda^{q_1}}\int_{\Omega}  \lVert u^N\rVert_{L_x^pL_t^q}^{q_1}d\mu_G^{(N)}(\phi_N)\lesssim \frac{(T^{1/q}\sqrt{q_1})^{q_1}}{\lambda^{q_1}}.
\end{align*}
Optimizing this estimate in $q_1$, we obtain ($\ref{7.9}$) as desired.
This completes the proof of Lemma $\ref{lem1a}$.
\end{proof}

We also need a variant of Lemma $\ref{lem1a}$ obtained by a similar argument.
\begin{lemma}
\label{lem4.1p}
Under the assumptions of Lemma $\ref{lem1a}$ and taking $M<N$, there is the distributional inequality
\begin{align}
\mu_F^{(N)}(\{\phi_N:\lVert u_N-P_Mu_N\rVert_{L_x^p(B;L_t^q([0,T]))}>\lambda\})<e^{-(\theta\lambda)^c}\label{4.3p}
\end{align}
with $\theta=T^{-\frac{1}{q}}M^{\frac{3}{p}-\frac{1}{2}}$.
\end{lemma}

The intent of ($\ref{4.3p}$) is to obtain better bounds letting $M\rightarrow\infty$ (since $p<6$).

\begin{proof}[Proof of Theorem $\ref{thm2}$.]
We first establish convergence of the sequence $(u^{N_k})$ with $N_k=2^k$.  Fix $0<s<1/2$ and $b>1/2$.  We compare the solutions $u_{N_0}$, $u_{N_1}$, $N_0<N_1$, on a fixed a time interval $[0,T]$ with $T>0$ arbitrary.  

We begin by constructing a suitably large set of initial data for which the global existence result will be shown.  Fix $1\leq p<6$ sufficiently close to $6$ and $q$ large to be specified later in the argument.  For each dyadic integer $N_0\geq 1$, let $0<B(N_0)<N_0$ be a fixed parameter to be determined later in the argument, and for dyadic $N_1>N_0\geq 1$, define the set
\begin{align}
\nonumber &\Omega_{N_0,N_1}:=\bigg\{\omega\in \Omega: \lVert \phi_{N_1}-\phi_{N_0}\rVert_{H^s}<N_0^{-\frac{1}{2}(\frac{1}{2}-s)},\\
&\hspace{1.5in}\lVert u_{N_0}\rVert_{L_x^p(B;L_{t}^{q}(0,T))},\, \lVert u_{N_1}\rVert_{L_x^p(B;L_{t}^{q}(0,T))}<B(N_0)\bigg\}.\label{def-omega}
\end{align}
By the probabilistic estimates of Lemma \ref{lem-prob-data} and Lemma \ref{lem1a}, we immediately have
\begin{align*}
\mu_F\bigg(\bigg\{\phi_\omega:\omega\in \Omega\setminus\Omega_{N_0,N_1}\bigg\}\bigg)\lesssim \exp(-B(N_0)^c),\quad 1\leq N_0<N_1.
\end{align*}

Subdivide $[0,T]$ in intervals of size $\Delta t$ (which will depend on $B(N_0)$), and write
\begin{align}
\nonumber &u_{N_1}(t+\Delta t)-u_{N_0}(t+\Delta t)=\\
&\hspace{1.2in} S(\Delta t)(u_{N_1}(t)-u_{N_0}(t))\label{7.18}\\
&\hspace{1in}-i\int_t^{t+\Delta t} S(t+\Delta t-\tau)(\sqrt{-\Delta})^{-1}f(\tau)d\tau\label{7.19}
\end{align}
where
\begin{align*}
f=P_{N_1}(|\iRe u^{N_1}|^\alpha \iRe u^{N_1})-P_{N_0}(|\iRe u^{N_0}|^\alpha \iRe u^{N_0}).
\end{align*}

In what follows, we denote $\lVert \cdot\rVert_{X^{s,b}([t,t+\Delta t])}$ by $\lVert \cdot\rVert_{s,b}$.  Note that since $b>\frac{1}{2}$,
\begin{align}
\lVert u_{N_1}(t+\Delta t)-u_{N_0}(t+\Delta t)\rVert_{H^s}\lesssim \lVert u_{N_1}-u_{N_0}\rVert_{s,b}.\label{7.20}
\end{align}

The $\lVert \cdot\rVert_{s,b}$ norm of ($\ref{7.18}$) is then immediately bounded by
\begin{align}
\lVert u_{N_1}(t)-u_{N_0}(t)\rVert_{H_x^s}.\label{7.21}
\end{align}

To bound the $\lVert \cdot \rVert_{s,b}$ norm of ($\ref{7.19}$), decompose $f$ as
\begin{align}
f&=(P_{N_1}-P_{N_0})(|\iRe u_{N_1}|^\alpha \iRe u_{N_1})\label{7.22}\\
&\hspace{0.4in}+P_{N_0}(|\iRe u_{N_1}|^\alpha \iRe u_{N_1}-|\iRe u_{N_0}|^\alpha \iRe u_{N_0}).\label{7.23}
\end{align}

Fix $\sigma>\frac{1}{2}$.  The $\lVert \cdot\rVert_{s,b}$ norm of the contribution of ($\ref{7.22}$) in ($\ref{7.19}$) is then estimated by invoking Lemma $\ref{lem210}$, giving
\begin{align}
\nonumber &\bigg\lVert \int_{t_0}^t S(t-\tau)(\sqrt{-\Delta})^{-1}(P_{N_1}-P_{N_0})(|\iRe u_{N_1}|^\alpha \iRe u_{N_1})d\tau\bigg\rVert_{X^{s,b}(J)}\\
\nonumber &\hspace{1.2in}\lesssim N_0^{-(\sigma-s)}\lVert |u_{N_1}|^{\alpha+1}\rVert_{L_x^{\frac{3}{3-\sigma}+}L_t^2}\\
&\hspace{1.2in}\lesssim N_0^{-(\sigma-s)}B(N_0)^{\alpha+1}\label{eqc2}
\end{align}
provided
\begin{align}
\frac{1}{2}<\sigma<\frac{5-\alpha}{2}.\label{7.15}
\end{align}

On the other hand, turning to the $\lVert\cdot\rVert_{s,b}$ norm of the contribution of ($\ref{7.23}$), note that another application of Lemma $\ref{lem210}$ gives
\begin{align}
\nonumber &\bigg\lVert \int_{t_0}^t S(t-\tau)(\sqrt{-\Delta})^{-1}P_{N_0}(|\iRe u_{N_1}|^\alpha \iRe u_{N_1}-|\iRe u_{N_0}|^\alpha \iRe u_{N_0})d\tau\bigg\rVert_{X^{\sigma,b}(J)}\\
\nonumber &\hspace{0.7in}\lesssim \lVert |u_{N_1}-u_{N_0}|(|u_{N_0}|^\alpha+|u_{N_1}|^\alpha)\rVert_{L_x^{\frac{3}{3-\sigma}+}L_t^2}\\
\nonumber &\hspace{0.7in}\lesssim (\Delta t)^\gamma\lVert u_{N_1}-u_{N_0}\rVert_{L_x^{p_1}L_t^{q_1}}(\lVert u_{N_0}\rVert_{L_x^pL_t^q}^\alpha+\lVert u_{N_1}\rVert_{L_x^pL_t^q}^\alpha)\\
&\hspace{0.7in}\lesssim (\Delta t)^\gamma B(N_0)^\alpha\lVert u_{N_1}-u_{N_0}\rVert_{X^{s,b}(J)},\label{eqc3}
\end{align}
for all $\omega\in \Omega_{N_0,N_1}$, with the exponents $p_1\geq 1$ and $q_1\geq 1$ satisfying $$\frac{3-\sigma}{3}>\frac{1}{p_1}+\frac{\alpha}{p}\quad\textrm{and}\quad \frac{1}{2}=\frac{1}{q_1}+\frac{\alpha}{q}+\gamma,$$ and provided that
\begin{align*}
\frac{3}{2}-\frac{1}{q_1}-\frac{3}{p_1}<s.
\end{align*}
Note that the choice of $p_1$ and $q_1$ is possible under the condition
\begin{align*}
0<\gamma<2-\frac{\alpha}{2}-\frac{\alpha}{q}.
\end{align*}
Letting $\alpha<4$ be given and choosing $q$ sufficiently large ensures that such a $\gamma$ exists.

Combining ($\ref{7.21}$), ($\ref{eqc2}$) and ($\ref{eqc3}$), we obtain
\begin{align}
\nonumber \lVert u_{N_1}-u_{N_0}\rVert_{X^{s,b}(J)}&\lesssim \lVert u_{N_1}(t_0)-u_{N_0}(t_0)\rVert_{H_x^s}+B(N_0)^{\alpha+1}N_0^{-(\sigma-s)}\\
&\hspace{0.6in}+(\Delta t)^\gamma B(N_0)^\alpha\lVert u_{N_1}-u_{N_0}\rVert_{X^{s,b}(J)}\label{eqcc1}
\end{align}

Choose
\begin{align}
\Delta t=\bigg[\frac{1}{2CB(N_0)^\alpha}\bigg]^{1/\gamma}\label{def_delta}
\end{align}
and
\begin{align}
B(N_0)\sim (\log N_0)^{\gamma/\alpha}\label{4.20}
\end{align}
where $C>0$ is the implicit constant in ($\ref{eqcc1}$), so that $\exp\left(\frac{T}{\Delta t}\right)$ is a sufficiently small power of $N_0$.  

We get the bound
\begin{align}
\lVert u_{N_1}-u_{N_0}\rVert_{X^{s,b}(J)}&\lesssim \lVert u_{N_1}(t_0)-u_{N_0}(t_0)\rVert_{H_x^s}+N_0^{-(\sigma-s)}B(N_0)^{\alpha+1}.\label{7.32}
\end{align}

Recalling ($\ref{7.20}$), it follows that 
\begin{align}
\nonumber\lVert (u_{N_1}-u_{N_0})(t+\Delta t)\rVert_{H_x^s}&\lesssim \lVert u_{N_1}-u_{N_0}\rVert_{s,b}\\
&\leq C\lVert (u_{N_1}-u_{N_0})(t)\rVert_{H_x^s}+N_0^{-\frac{\sigma-s}{2}}\label{7.33}
\end{align} 
For all $\omega\in \Omega_{N_0,N_1}$ we have
\begin{align*}
\lVert (u_{N_1}-u_{N_0}(0)\rVert_{H_x^s}<N_0^{-\frac{1}{2}(\frac{1}{2}-s)}
\end{align*}
so that an iterative application of ($\ref{7.32}$) and ($\ref{7.33}$) give, by the choice of $\Delta t$,
\begin{align}
\max_{t\in [0,T)} \lVert (u_{N_1}-u_{N_0})(t)\rVert_{H_x^s}&\leq C^{T/(\Delta t)}N_0^{-\frac{1}{2}(\frac{1}{2}-s)}<N_0^{-\frac{1}{4}(\frac{1}{2}-s)}.\label{eq-n1n0}
\end{align}

From ($\ref{7.32}$) and ($\ref{eq-n1n0}$) we obtain
\begin{align}
\lVert u_{N_1}-u_{N_0}\rVert_{X^{s,b}([t,t+\Delta t])}\lesssim N_0^{-\frac{1}{4}(\frac{1}{2}-s)}\label{7.35}
\end{align}
for all subintervals $[t,t+\Delta t]\subset [0,T]$, from where one easily derives that
\begin{align}
\lVert u_{N_1}-u_{N_0}\rVert_{X^{s,b}([0,T])}\lesssim \frac{T^{1/2}}{(\Delta t)^{\frac{1}{2}+b}}N_0^{-\frac{1-2s}{8}}\lesssim N_0^{-\frac{1-2s}{16}}\label{7.36}
\end{align}
for every $\omega\in \Omega_{N_0,N_1}$.  

Let $(N_k)$ be a given sequence of integers growing sufficiently fast, e.g. $N_k=2^k$.  Arguing as above, we obtain that for every $\omega\in \Omega_{N_k,N_{k+1}}$ with $\Omega_{N_k,N_{k+1}}$ defined as in ($\ref{def-omega}$),
\begin{align*}
\lVert u^{N_{k+1}}-u^{N_k}\rVert_{X^{s,b}([0,T])}&\lesssim N_k^{-\frac{1}{2}(\sigma-s)}.
\end{align*}

To finish the proof of Theorem $\ref{thm2}$, let $\Omega'$ be the set
\begin{align*}
\Omega':=\bigcup_{K\geq 1}\bigcap_{k\geq K} \Omega_{N_k,N_{k+1}}.
\end{align*}
We then obtain that the sequence $(u^{N_k})$ is a Cauchy sequence in the space $X^{s,b}\subset C_tH_x^s$ for all $\omega\in \Omega'$.  Moreover, elementary inequalities combined with the estimate 
\begin{align*}
\mu_F\bigg(\bigg\{\phi_\omega:\omega\in \Omega\setminus\Omega_{N_k,N_{k+1}}\bigg\}\bigg)\lesssim \exp(-B(N_k)^c).\end{align*}
give
\begin{align*}
\mu_F(\{\phi_{\omega}:\omega\in \Omega\setminus\Omega'\})=0.
\end{align*}
The convergence of $u^{N_k}$ to some $u_*$ therefore holds $\mu_F$-almost surely.  

The choice of $B(N_0)$ in ($\ref{4.20}$) (as needed to get and appropriate $\Delta t$) leads to measure estimates in the above argument which do not suffice to conclude immediately convergence of the full sequence $(u_N)$.  But it is possible to achieve this by the following variant of the preceeding.  Given large $N_0$, take
\begin{align}
M=M_0=(\log N_0)^{C(p)}\label{4.26}
\end{align}
with a suitable power $C(p)$.  Replace $\Omega_{N_0,N_1}$ by
\begin{align}
\nonumber \Omega_{N_0}=\bigg\{\omega\in \Omega&: \lVert \phi-\phi_{N_0}\rVert_{H_x^s}<N_0^{-\frac{1}{2}(\frac{1}{2}-s)},\quad \lVert u_{N_0}\rVert_{L_x^p(B;L_t^{q}(0,T))}<B(N_0)\\
&\hspace{0.4in}\textrm{and}\quad\max_{N_0\leq N\leq 2N_0}\,\,\lVert u_N-P_Mu_N\rVert_{L_x^p(B;L_t^q(0,T))}<1\bigg\}\label{4.27}
\end{align}
with $B(N_0)$ as above.  Thus, by Lemmas $\ref{lem1a}$, $\ref{lem4.1p}$ ($T$ is fixed),
\begin{align*}
\mu_F(\Omega\setminus \Omega_{N_0})<e^{-B(N_0)^c}+N_0e^{-\theta^c}
\end{align*}
with $\theta\sim M^{\frac{3}{p}-\frac{1}{2}}$.

Proper choice of $C(p)$ in ($\ref{4.26}$) gives 
\begin{align}
\mu_F(\Omega\setminus \Omega_{N_0})<2e^{-B(N_0)^c}.\label{4.28}
\end{align}

For $N_0\leq N_1\leq 2N_0$, estimate
\begin{align}
\nonumber \lVert u_{N_1}\rVert_{L_x^pL_t^q}&\leq \lVert u_{N_0}\rVert_{L_x^pL_t^q}+\lVert u_{N_1}-P_Mu_{N_1}\rVert_{L_x^pL_t^q}\\
\nonumber &\hspace{0.4in}+\lVert u_{N_0}-P_Mu_{N_0}\rVert_{L_x^pL_t^q}+\lVert P_M(u_{N_1}-u_{N_0})\rVert_{L_x^pL_t^q}\\
\nonumber &<B(N_0)+2+T^{\frac{1}{q}}\lVert P_M(u_{N_1}-u_{N_0})\rVert_{L_{t,x}^\infty}\\
\nonumber &<B(N_0)+2+T^{\frac{1}{q}}M^{\frac{3}{2}-s}\lVert u_{N_1}-u_{N_0}\rVert_{s,b}\\
&<(\log N_0)^C\lVert u_{N_1}-u_{N_0}\rVert_{s,b}+2B(N_0)\label{4.29}
\end{align}

Using ($\ref{4.29}$), inequality ($\ref{eqcc1}$) becomes
\begin{align}
\nonumber \lVert u_{N_1}-u_{N_0}\rVert_{s,b}&\lesssim \lVert u_{N_1}(t_0)-u_{N_0}(t_0)\rVert_{H_x^s}\\
\nonumber &\hspace{0.4in}+(\Delta t)^\gamma B(N_0)^\alpha\lVert u_{N_1}-u_{N_0}\rVert_{s,b}\\
\nonumber &\hspace{0.4in}+(\log N_0)^{C\alpha}\lVert u_{N_1}-u_{N_0}\rVert_{s,b}^{\alpha+1}\\
&\hspace{0.4in}+N_0^{-(\sigma-s)}B(N_0)^{\alpha+1}.\label{4.30}
\end{align}

Therefore
\begin{align*}
\lVert u_{N_1}-u_{N_0}\rVert_{s,b}&\lesssim \lVert u_{N_1}(t_0)-u_{N_0}(t_0)\rVert_{H_x^s}\\
&\hspace{0.4in}+(\log N_0)^{C\alpha}\lVert u_{N_1}-u_{N_0}\rVert_{s,b}^{\alpha+1}\\
&\hspace{0.4in}+N_0^{-\frac{1}{2}(\sigma-s)}
\end{align*}
which again implies ($\ref{7.33}$) and ($\ref{eq-n1n0}$).
\end{proof}

As a concluding observation, we give the following refined bound on the growth of $L_t^\infty H_x^s$ norms for the solutions constructed in Theorem $\ref{thm2}$.
\begin{proposition}
\label{prop-concluding}
Fix $0<\sigma<\frac{5-\alpha}{2}$ and $0<T<+\infty$.  Then there exists a set $\Sigma=\Sigma(\delta)\subset\Omega$ such that $\textrm{mes}(\Sigma)>1-\delta$, and for $\omega\in \Sigma$ the solution $u=u^N$ of (NLW) satisfies
\begin{align*}
\lVert u(t)-S(t)\phi\rVert_{L_t^\infty([0,T];H_x^\sigma(B))}\leq C\bigg(\log \frac{T}{\delta}\bigg)^{C}.
\end{align*}
\end{proposition}

The proof of this claim works in an identical fashion as in the proof of Proposition $\ref{prop-fp-2}$.

\end{document}